\documentclass[12pt]{article}
\usepackage{amsmath, amssymb, amsfonts}
\usepackage{graphicx}
\usepackage{enumerate}
\usepackage{breqn}
\usepackage{caption,subcaption}
\usepackage[colorlinks=true,citecolor=black,linkcolor=black,urlcolor=blue]{hyperref}
\usepackage{amsthm}
\usepackage{cleveref}

\newtheorem{theorem}{Theorem}[section]
\newtheorem{lemma}[theorem]{Lemma}
\newtheorem{cor}[theorem]{Corollary}

\numberwithin{equation}{section}
\numberwithin{figure}{section}

\newcommand{\arxiv}[1]{\href{http://arxiv.org/abs/#1}{\texttt{arXiv:#1}}}

\title{\bf Enumeration of lozenge tilings of halved hexagons with a boundary defect}
\author{Ranjan Rohatgi\\
\small Department of Mathematics\\[-0.8ex]
\small Indiana University\\[-0.8ex] 
\small Bloomington, IN, 47405\\
\small\tt rrohatgi@indiana.edu\\}
\date{}

\begin{document}

\maketitle

 \begin{abstract}
 We generalize a special case of a theorem of Proctor on the enumeration of lozenge tilings of a hexagon with a maximal staircase removed, using Kuo's graphical condensation method.  Additionally, we prove a formula for a weighted version of the given region.  The result also extends work of Ciucu and Fischer.  By applying the factorization theorem of Ciucu, we are also able to generalize a special case of MacMahon's boxed plane partition formula.  
 \end{abstract}

  \section{Introduction}
 
 The triangular lattice is the tiling of the plane by unit equilateral triangles.  Without loss of generality, we assume that the lattice comprises horizontal lines, as well as lines whose angles of incidence to the horizontal lines is either 60 or 120 degrees.  A \textit{region} in the triangular lattice is any finite union of these unit triangles and a \textit{lozenge} is any union of two unit triangles which share an edge.  A \textit{lozenge tiling} of a region $R$ is any covering of all unit triangles in $R$ by non-overlapping lozenges.  It is clear that a region must be have the same number of upward-pointing unit triangles as downward-pointing ones to have any tilings at all, since a lozenge contains one unit triangle of each type.  We say that such a region is \textit{balanced}.  We can assign to any lozenge that could be used in a tiling a  weight, $w$, which is a positive real number.  An \textit{unweighted} region has all weights equal to 1.    
 
The weight of a lozenge tiling of $R$ is the product of all the weights of the lozenges used in the tiling.  We denote by $M(R)$ the \textit{matching generating function} of the region $R$, which is the sum of the weights of all tilings of $R$.  For an unweighted region, the matching generating function simply gives the number of tilings of the region.
 
 MacMahon's work in \cite{macmahon} proved that for a hexagon with side-lengths $b,c,d,b,c,d$, the number of lozenge tilings is given by the formula 
 \begin{equation}\label{macmahonthm}
\frac{H(b)H(c)H(d)H(b+c+d)}{H(b+c)H(b+d)H(c+d)},
 \end{equation}
where we define the hyper factorials $H(n)$ for positive integers $n$ by
\begin{equation}
H(n):=0!1!\ldots (n-1)!
\end{equation}

The simplicity of (\ref{macmahonthm}) has inspired many to look for generalizations or similar results.  We present ours in the next section.

\section{Statement of Main Results}

We define two regions, $R_{a,k,j,x}$ and $R'_{a,k,j,x}$, the latter of which is a weighted version of the former.  The north edge of each region has length $x$, followed by a northeast edge of length $a+2k$, a southeast edge of length $a$, and a south edge of length $x+k$.  Finally, we close the regions by connecting the west endpoints of the north and south edges via a zigzag line whose unit edges alternate northwest and northeast.  This zigzag line comprises $2a+2k$ unit segments, or $a+k$ ``bumps."  To balance the region, we remove $k$ consecutive upward-pointing unit triangles from the northeast side, after leaving a gap of $j-1$ unit triangles.  It is evident that $j\in\{1,2,\ldots,a+k+1\}$.  In the pictures below we've removed the forced lozenges due to the ``spikes" on the northeast side.  This unweighted region is $R_{a,k,j,x}$.  If we weight each of the vertical lozenges in the ``bumps" by a factor of $\frac{1}{2}$ we have $R'_{a,k,j,x}$.  In Figure~\ref{rak}, the lozenges with ovals are the weighted ones.  
 
 \begin{figure}[h]
\centering
\begin{subfigure}{.45\textwidth}
\centering
\includegraphics[height=2.5in]{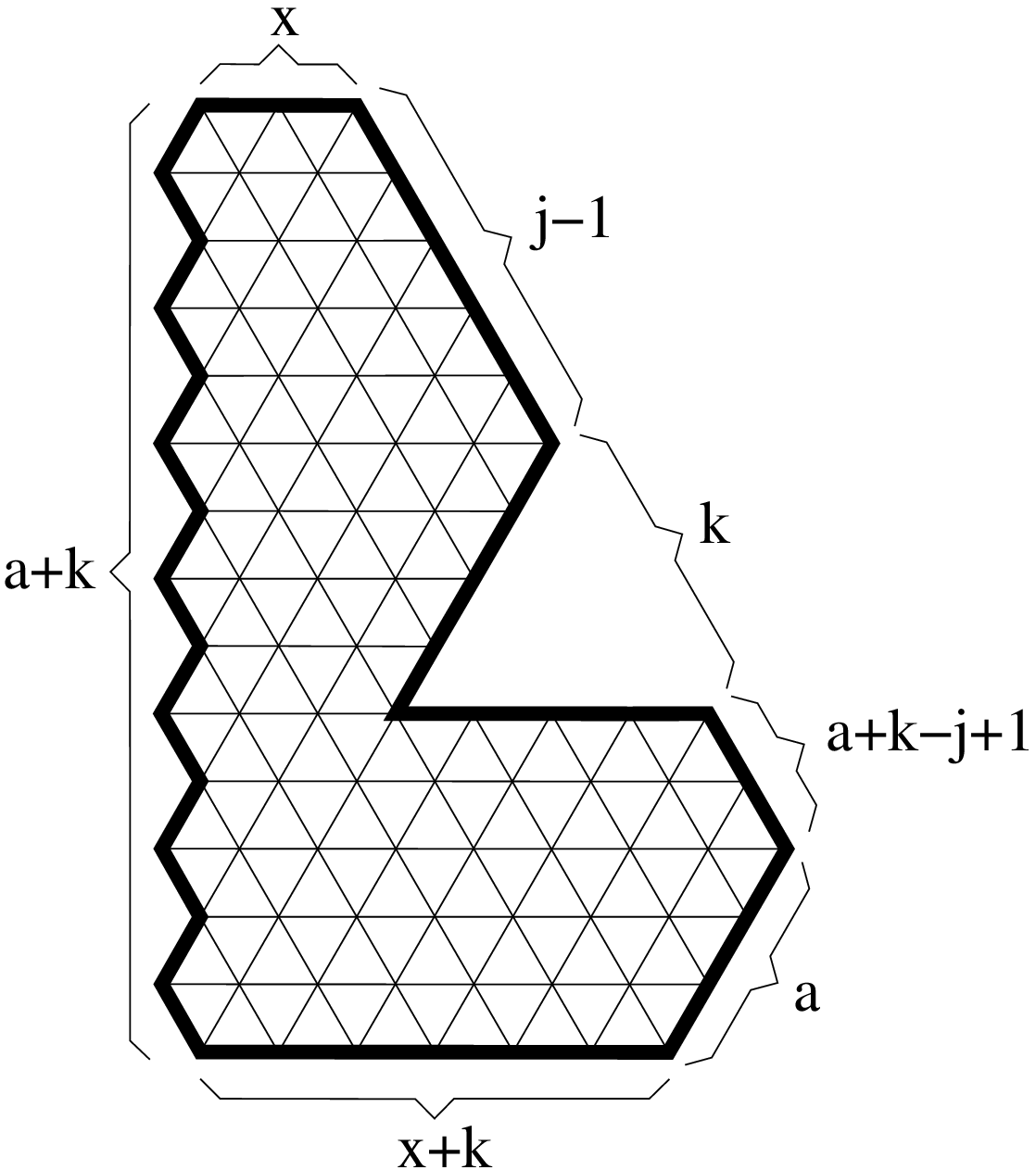}
\caption{$R_{a,k,j,x}$}
\end{subfigure}
\begin{subfigure}{.45\textwidth}
\centering
\includegraphics[height=2.5in]{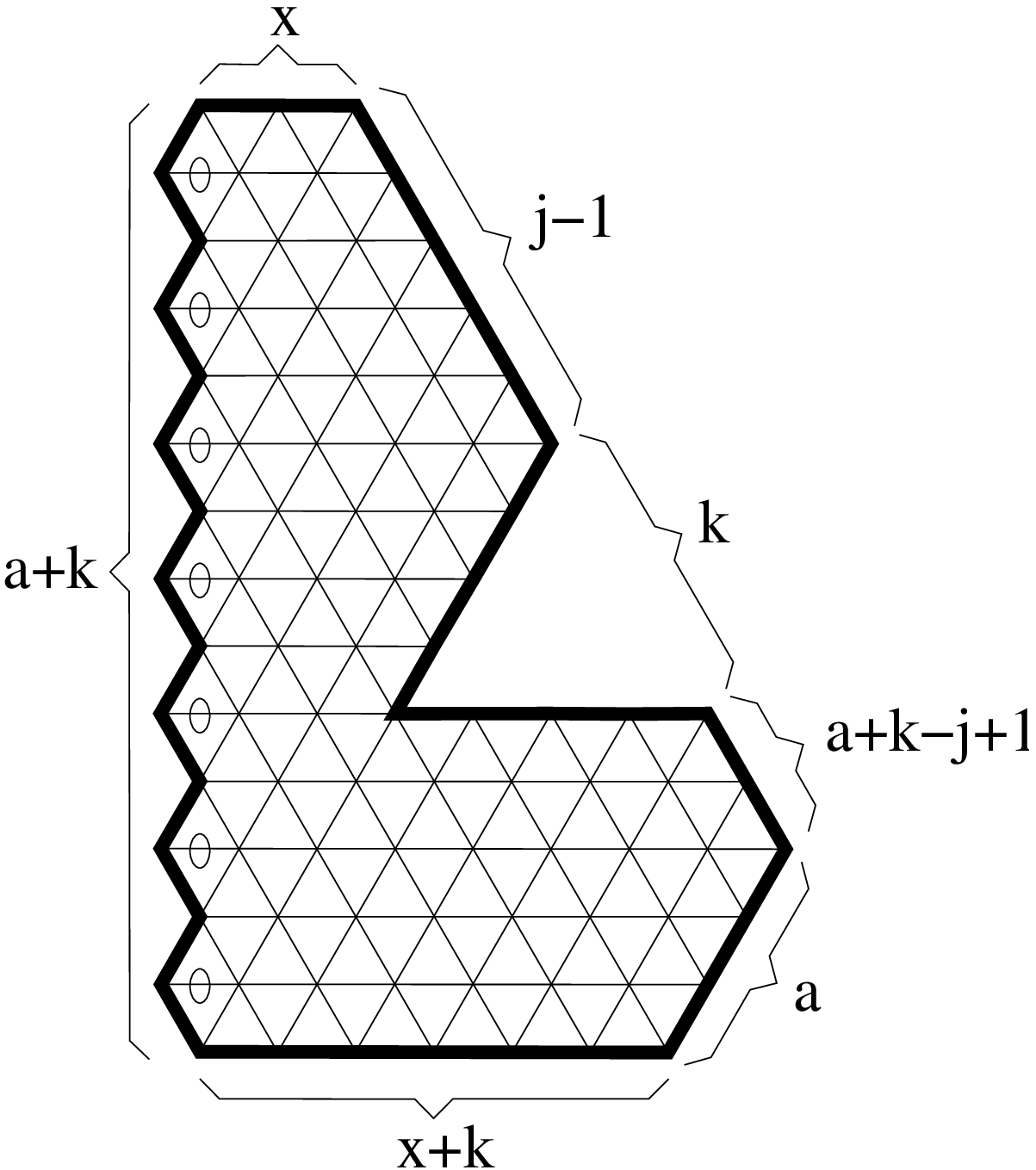}
\caption{$R'_{a,k,j,x}$}
\end{subfigure}
\caption{} 
\label{rak}
\end{figure}

The region $R_{a,k,jx}$ extends previous work in two ways.  Setting $k=0$ gives us a symmetric region with no unit triangles removed on the northeast side. In \cite{proctor}, Proctor generalizes this region  by extending the northwest side; here we generalize such a region by introducing a boundary defect on the northeast side.  In \cite{ciucufischer13}, Ciucu and Fischer enumerate the tilings of a region they denote $R_{x,a,k}$.  The region is identical to $R_{a,k,j,x}$ except that the position of the removed triangles on the northeast side is fixed at $j=a+k$. 

\begin{theorem}\label{rakjxthmexplicit}
\begin{multline*}
M(R_{a,k,j,x})=\left[\prod_{n=1}^a\left(\frac{x+k+n}{k+n}\right)^{f_{a}(n)}\right]\left[\prod_{n=1}^{a-1}\left(\frac{2x+2k+2n+1}{2k+2n+1}\right)^{f_{a-1}(n)}\right]\\
\times\left[\prod_{n=1}^{j-k-1}\left(\frac{k+n}{x+k+n}\right)^{f_{j-k-1}(n)}\right]\left[\prod_{n=1}^{j-k-2}\left(\frac{2k+2n+1}{2x+2k+2n+1}\right)^{f_{j-k-2}(n)}\right]\\
\times\left[\prod_{n=1}^{j-1}\frac{x+n}{2x+n}\right]\left[\prod_{1\leq n\leq m\leq j-1}\frac{2x+n+m-1}{n+m-1}\right]\\
\times\prod_{n=1}^{a-j+k+1}\left[\prod_{m=1}^{j-k}\frac{k+n+m-1}{n+m-1}\prod_{m=j-k+1}^{j-k-1+n}\frac{2k+n+m-1}{n+m-1}\right],
\end{multline*}
where $f_a(i)=\frac{a+1}{2}-|\frac{a+1}{2}-i|$.
\end{theorem}

\begin{theorem}\label{weightedrakjxthmexplicit}
\begin{multline*}
M(R'_{a,k,j,x})=\left[\prod_{n=1}^a\left(\frac{2x+2k+2n-1}{2k+2n-1}\right)^{f_{a}(n)}\right]\left[\prod_{n=1}^{a-1}\left(\frac{x+k+n}{k+n}\right)^{f_{a-1}(n)}\right]\\
\times\left[\prod_{n=1}^{j-k-1}\left(\frac{2k+2n-1}{2x+2k+2n-1}\right)^{f_{j-k-1}(n)}\right]\left[\prod_{n=1}^{j-k-2}\left(\frac{k+n}{x+k+n}\right)^{f_{j-k-2}(n)}\right]\\
\times\left[\prod_{1\leq n\leq m\leq j-1}\frac{2x+n+m-1}{n+m-1}\right]\left[\prod_{n=1}^{a-j+k+1}\frac{k+j+n-1}{j+n-1}\right]\\
\times\frac{1}{2^{a+k}}\prod_{n=1}^{a-j+k+1}\left[\prod_{m=1}^{j-k}\frac{k+n+m-1}{n+m-1}\prod_{m=j-k+1}^{j-k-1+n}\frac{2k+n+m-1}{n+m-1}\right],
\end{multline*}
with $f_a(i)$ as above.
\end{theorem}
 
 Using \Cref{rakjxthmexplicit,weightedrakjxthmexplicit}, we can also prove a formula for the number of tilings of the region described below.  
 
 Consider a hexagon with side-lengths $b, c+2k, c, b+2k, c, c+2k$ (again, starting with the north side). We must remove $2k$ upward-pointing unit triangles to balance the region (or we may remove triangles of both types, but with $2k$ greater upward-pointing ones).  We will remove $k$ consecutive unit triangles from the northeast side and the corresponding ones from the northwest side.  Denote such a region by $DDH_{b,c,2k,j}$.  The index $j$ tells us the precise location of the removed unit triangles, just as in the region $R_{a,k,j,x}$.  Figure~\ref{ddh4763} shows an example of such a region, with $b=4, c=7, k=3,\textrm{ and } j=3$.

  \begin{figure}
 \centering
\includegraphics[height=3.2in]{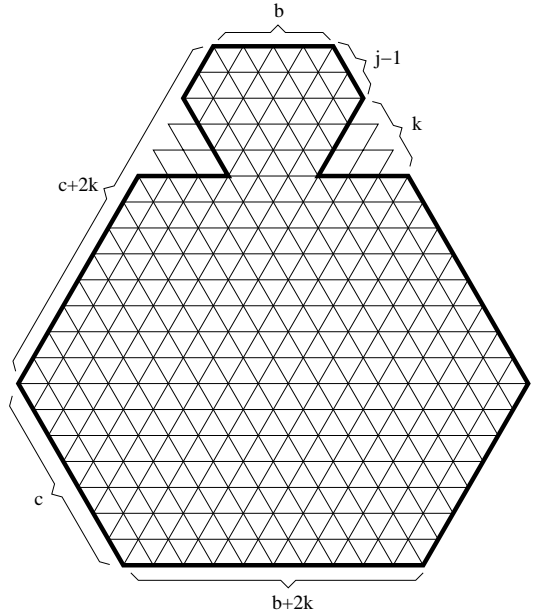}
\caption{The $k$ removed unit triangles on the northeast and northwest sides leave behind ``spikes."  In any lozenge tiling of $DDH(b,c,2k,j)$, there is only one way to tile these.}
\label{ddh4763}
\end{figure}

Though it can be much more widely applied, Ciucu's factorization theorem from \cite{ciucu97} provides a method to enumerate the tilings of a symmetric region on the triangular lattice by computing the matching generating functions of one weighted and one unweighted subregion induced by cutting the region in half.
We apply the theorem to $DDH$ regions as in Figure~\ref{factthmapplied}.  This immediately gives the following corollary, which is a generalization of a special case of a result of Lai \cite{lai}.  We have called the $R$- and $R'$-type regions ``halved hexagons with boundary defects" since that is exactly the role they play in determining the number of tilings of $DDH$ regions.

 \begin{figure}
\centering
\includegraphics[height=2in]{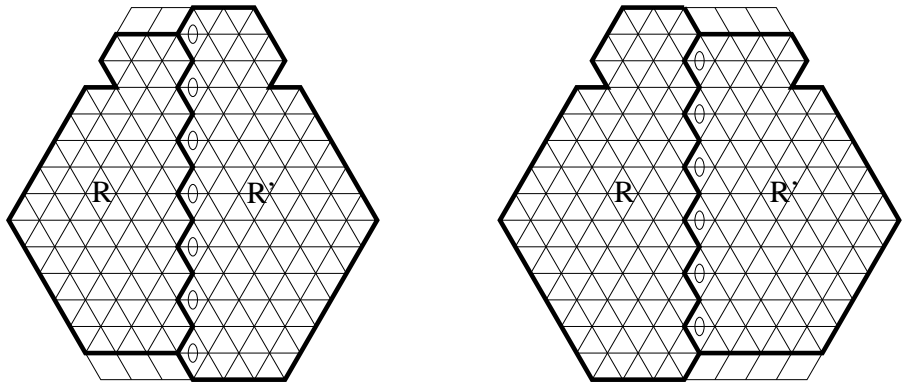}
\caption{The subregions $R$ and $R'$, after forcing, obtained by applying the factorization theorem to a $DDH$ region with $b$ even (left) and $b$ odd (right).}
\label{factthmapplied}
\end{figure}

\begin{cor}
\label{mainthm}
\begin{equation}
M(DDH_{b,c,2k,j})=
\begin{cases} 2^{c+k} M(R_{c-1,k,j-1,b/2}) M(R'_{c,k,j,b/2}) & \text{if $b$ is even,} \\
2^{c+k} M(R_{c,k,j,(b-1)/2})M(R'_{c-1,k,j-1,(b+1)/2}) &\text{if $b$ is odd.}
\end{cases}
\end{equation}
\end{cor}

 \section{Preliminaries}

 The \textit{dual graph} of a region $R$ is the graph comprising one vertex for each unit triangle in $R$.  Two vertices share an edge in the dual graph if and only if their corresponding unit triangles are edge-adjacent.  An edge in the dual graph has weight $w$ precisely if the corresponding lozenge in $R$ also did.  For regions on the triangular lattice, we've seen that each unit triangle is either pointing upwards or downwards - in particular, there are two types of unit triangles.  The resulting graph is now bipartite, and lozenge tilings of a region $R$ are clearly in one-to-one correspondence with perfect matchings of the bipartite dual graph (for a weighted region, the matching generating functions coincide).
 
In \cite{ciuculai14}, Ciucu and Lai give conditions under which the matching generating function of a bipartite graph is the product of the matching generating function of two induced subgraphs.  We extend it slightly.  

\begin{lemma}[Graph Splitting Lemma]\label{graphsplit}
Let $G=(V_1,V_2,E)$ be a bipartite graph.  Assume $H$ is an induced subgraph of $G$ that satisfies the following condition:
\begin{enumerate}[(i)]
\item (Separating Condition) There are no edges of $G$ connecting a vertex in $V(H)\cap V_1$ and a vertex in $V(G-H)$.
\item $|V(H)\cap V_1|\geq|V(H)\cap V_2|$.
\end{enumerate}
Then $$M(G)=M(H)M(G-H).$$
\end{lemma}
 \begin{proof}
 If $|V(H)\cap V_1|=|V(H)\cap V_2|$, then \cite{ciuculai14} provides the proof.
 Suppose $|V(H)\cap V_1|>|V(H)\cap V_2|$, in which case $M(H)=0.$  We must show $M(G)=0.$  By the separating condition, there is no edge in $G$ connecting a vertex in  $V(H)\cap V_1$ to a vertex in $V(G-H)$.  In a perfect matching of $G$, every vertex in $V(H)\cap V_1$ must then be connected to a vertex in $V(H)\cap V_2$, but there are not enough such vertices by assumption.  Hence $M(G)=0.$
 \end{proof}
 
We say that we \textit{cut} a graph (or region) into subgraphs (or subregions).

Although it is more general, Kuo's graphical condensation method \cite{kuo} can be used to count matchings of bipartite graphs.  There are several versions; the one we will use is stated below.

\begin{theorem}\label{kuograph}
 Let $G=(V_1,V_2,E)$ be a plane bipartite graph with $|V_1|=|V_2|+1,$ and suppose that vertices $t,u,v,\textrm{and }w$ appear cyclically on a face of $G$.  If $t,u,v\in V_1$ and $w\in V_2$, then 
 \begin{multline*}
 M(G-u)M(G-\{t,v,w\})=\\M(G-t)M(G-\{u,v,w\})+M(G-v)M(G-\{t,u,w\}).
 \end{multline*}
 \end{theorem}

The matching generating function for $R_{a,k,j,x}$, is very closely related to the number of tilings of a hexagon with side-lengths $c,a,b,c,a,b$ with a ``maximal staircase" removed.  We call such a region $P_{a,b,c}$ (see Figure~\ref{proctor}).  Here is the classical result due to Proctor \cite{proctor}.

\begin{theorem}[Proctor \cite{proctor}]\label{proctorthm}
For any non-negative integers $a,b,$ and $c$ with $a\leq b$, we have $$M(P_{a,b,c})=\prod_{i=1}^a \left[\prod_{j=1}^{b-a+1}\frac{c+i+j-1}{i+j-1}\prod_{j=b-a+2}^{b-a+i}\frac{2c+i+j-1}{i+j-1}\right],$$ where empty products are taken to be $1$.  Further, $M(P_{b+1,b,c})=M(P_{b,b,c}).$
\end{theorem}

\begin{figure}
\centering
\includegraphics[height=2.5in]{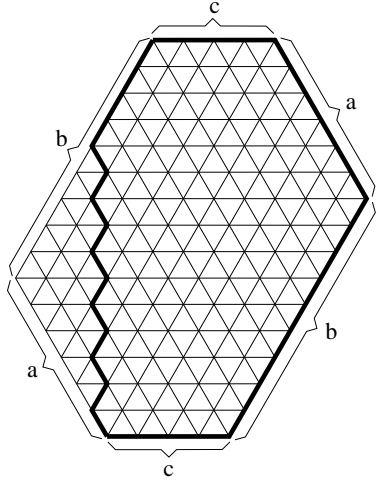}
\caption{The region $P_{a,b,c}$ for $a=6,b=9,\textrm{and } c=4.$}
\label{proctor}
\end{figure}

We will make use of the following corollary, in which $a=b$.  

\begin{cor}
For any non-negative integers $a$ and $c$ we have $$M(P_{a,a,c})=\prod_{i=1}^{a}\frac{c+i}{2c+i}\prod_{1\leq i\leq j\leq a} \frac{2c+i+j-1}{i+j-1}.$$  
\end{cor}

The family of regions $R_{a,k,j,x}$ extends the $P_{a,a,c}$ family: it is clear that $R_{a,0,j,x}=P_{a,a,x}.$

 The matching generating function for $R'_{a,k,j,x}$ requires a weighted version of \Cref{proctorthm}.  For the region $P'_{a,b,c}$, each lozenge that is part of the ``maximal staircase" has weight $\frac{1}{2}$, while the rest are unweighted (see Figure~\ref{proctorweight}). 
 
 \begin{figure}[ht]
 \centering
 \includegraphics[height=2.5in]{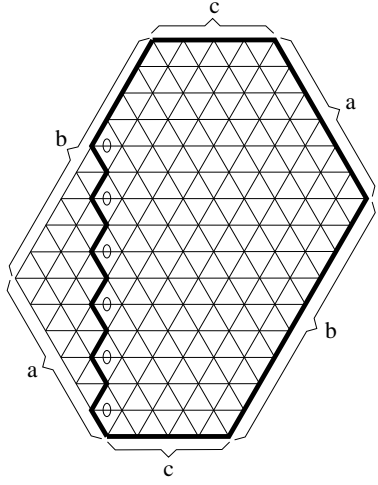}
 \caption{The region $P'_{a,b,c}$ has weighted lozenges along its west side.}
 \label{proctorweight}
 \end{figure}

In \cite{ck}, Ciucu calculated the matching generating function of a family of regions which include these weighted Proctor regions.

\begin{cor}
For any non-negative integers $a,b,$ and $c$ with $a\leq b$,  $$M(P'_{a,b,c})=\frac{M(P_{a,b,c})}{2^a}\cdot\prod_{i=1}^a\frac{2c+b-a+i}{c+b-a+i}.$$  As in the case of \Cref{proctorthm}, $M(P'_{b+1,b,c})=M(P'_{b,b,c}).$
\end{cor}

 We end this section with two more definitions.  For integers $a,k,$ and $x$ define 
 $$Q_{a,k,x}:=\prod_{i=1}^a (x+k+i)^{f_a(i)} \prod_{i=1}^{a-1} (2x+2k+2i+1)^{f_{a-1}(i)},$$ where empty products are again taken to be $1$.  We also define 
 $$Q'_{a,k,x}:=\prod_{i=1}^a(2x+2k+2i-1)^{f_a(i)}\prod_{i=1}^{a-1}(x+k+i)^{f_{a-1}(i)}.$$ 

\section{Proofs of \texorpdfstring{\Cref{rakjxthmexplicit,weightedrakjxthmexplicit}}{}}

We can rewrite \Cref{rakjxthmexplicit,weightedrakjxthmexplicit} as follows.

\begin{theorem}\label{rakjxthm}
For $a,k,\textrm{and } x$ non-negative integers and $j\in\{k,k+1,\ldots,\\a+k+1\}$, 
$$M(R_{a,k,j,x})=\frac{Q_{a,k,x}}{Q_{a,k,0}}\cdot\frac{Q_{j-k-1,k,0}}{Q_{j-k-1,k,x}}\cdot M(P_{j-1,j-1,x})\cdot M(P_{a-j+k+1,a,k}).$$
\end{theorem}

\begin{theorem}\label{weightedrakjxthm}
For $a,k,\textrm{and } x$ non-negative integers and $j\in\{k,k+1,\ldots,\\a+k+1\}$, 
$$M(R'_{a,k,j,x})=\frac{Q'_{a,k,x}}{Q'_{a,k,0}}\cdot\frac{Q'_{j-k-1,k,0}}{Q'_{j-k-1,k,x}}\cdot M(P'_{j-1,j-1,x})\cdot M(P'_{a-j+k+1,a,k}).$$
\end{theorem}

We will prove \Cref{rakjxthm}; the proof of \Cref{weightedrakjxthm} is similar.  

\begin{proof}
For $k>0$, it is clear that $j$ cannot be greater than $a+k+1$ if we are to remove $k$ triangles from the northeast side. Furthermore, if $j<k$, then \Cref{graphsplit} implies that $M(R_{a,k,j,x})=0$, where we take the induced subgraph $H$ to be the dual graph of the region above the cut (see Figure~\ref{graphsplitting}). 

\begin{figure}[h]
\centering
\includegraphics[height=2in]{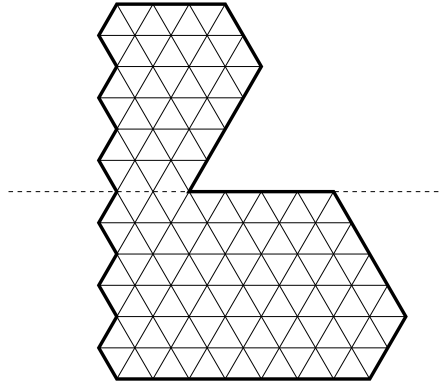}
\caption{The region $R_{2,4,3,3}$ has no tilings.}
\label{graphsplitting}
\end{figure}

We begin by proving \Cref{rakjxthm} for $j=k,k+1$.  In these two cases, we can apply \Cref{graphsplit} with a cut made on the south side of the boundary defect, as in Figure~\ref{kkplus1}.

\begin{figure}[h]
\centering
\includegraphics[height=2.5in]{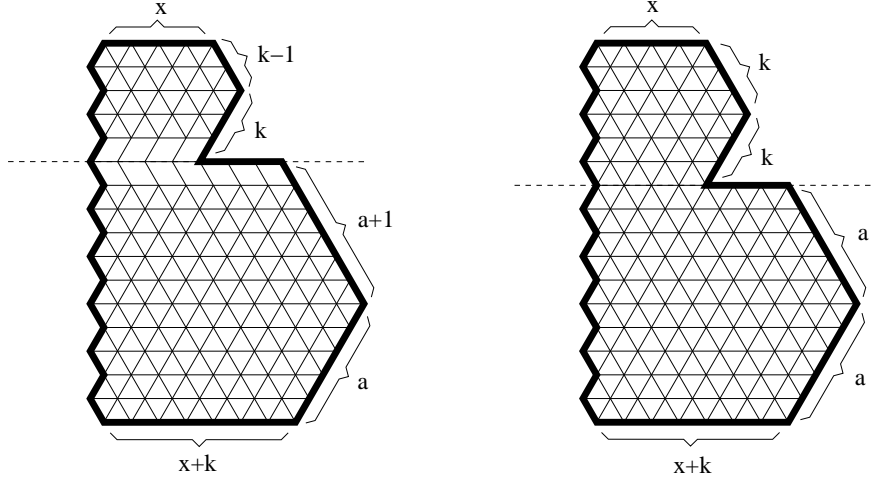}
\caption{The cuts when $j=k$ (left) or $j=k+1$ (right).}
\label{kkplus1}
\end{figure}

For the $j=k$ case, it is clear that
\[ M(R_{a,k,k,x})=M(P_{k-1,k-1,x})M(P_{a,a,x+k}).\]  Therefore, we must show that 
 \begin{equation}\label{jisk}
 M(P_{k-1,k-1,x})M(P_{a,a,x+k})=
 \frac{Q_{a,k,x}}{Q_{a,k,0}}\cdot\frac{Q_{-1,k,0}}{Q_{-1,k,x}}\cdot M(P_{k-1,k-1,x})\cdot M(P_{a+1,a,k}),
 \end{equation}
where the righthand side is obtained by plugging in $j=k$ into the claimed formula in \Cref{rakjxthm}.
We first need a lemma.

\begin{lemma}\label{j12}
For non-negative integers $a,k,\textrm{and } x$,
\begin{enumerate}[(i)]
	\item 
		\begin{multline*}
			Q_{a+1,k,x}=Q_{a,k,x}\cdot (x+k+\lceil\tfrac{a+2}{2}\rceil)(x+k+\lceil\tfrac{a+4}{2}\rceil)\ldots (x+k+a+1)\\
			\times (2x+2k+2\lfloor\tfrac{a+2}{2}\rfloor+1)(2x+2k+2\lfloor\tfrac{a+4}{2}\rfloor+1)\ldots (2x+2k+2a+1)					\end{multline*}
	\item 
		$$M(P_{a+1,a+1,x})=M(P_{a,a,x})\cdot \dfrac{x+a+1}{2x+a+1}\cdot \prod_{i=1}^{a+1}\dfrac{2x+a+i}{a				+i}$$
\end{enumerate}
\end{lemma}
\Cref{j12} is easily verified.

We can simplify (\ref{jisk}) by noting that two of the Q polynomials are taken to be 1 since their first index is negative.  Further, \Cref{proctorthm} (as well as the forcing in Figure~\ref{kkplus1}) shows $M(P_{a+1,a,k})=M(P_{a,a,k}).$

 \begin{lemma}\label{jisk2}
 For non-negative integers $a,k,$ and $x$, \[M(P_{a,a,x+k})=\frac{Q_{a,k,x}}{Q_{a,k,0}}\cdot M(P_{a,a,k}).\]
 \end{lemma}
\begin{proof}
We proceed by induction on $a$.  The result is clear for $a=0$ as both sides are 1.  Assuming it holds for $a$, we divide both sides of the claimed formula in \Cref{jisk2} with index $a+1$ by the same result with index $a$.  Using \Cref{j12}, and multiplying the numerator and denominator on the righthand side by $2^{\lfloor \frac{a+2}{2}\rfloor}$, we have simplified the proof to showing that

\begin{multline*}
\frac{x+k+a+1}{2x+2x+a+1}\cdot\prod_{i=1}^{a+1}\frac{2x+2k+a+i}{a+i}=\\
\frac{(2x+2k+a+2)(2x+2k+a+3)\ldots (2x+2k+2a+2)}{(2k+a+2)(2k+a+3)\ldots (2k+2a+2)}\\
\times\frac{k+a+1}{2k+a+1}\prod_{i=1}^{a+1}\frac{2k+a+i}{a+i}
\end{multline*}

This is easily checked.
\end{proof} 

The $j=k+1$ case is nearly identical, with \Cref{jisk2} proving that the formula for its number of tilings using \Cref{graphsplit} matches the claimed formula in \Cref{rakjxthm}.  It is important to point that these two cases also prove  \Cref{rakjxthm} for $j=1$ and $j=2$ based upon the possible values for $k$, as we will assume $j>2$ from now on.

 For $j>k+1$, we apply \Cref{kuograph} to a region slightly different from $R_{a,k,j,x}.$  Instead of removing a run of $k$ consecutive upward-facing unit triangles starting at position $j$ from the northeast side, only remove $k-1$ such triangles from position $j+1$.  Figure~\ref{rkuo} shows the locations of $t,u,v,\textrm{and }w,$ all on the outside face of the dual graph.  Notice that $w$ is pointing downwards while $t,u,\textrm{and }v$ are pointing upwards.
 
 \begin{figure}[h]
 \centering
 \includegraphics[height=2in]{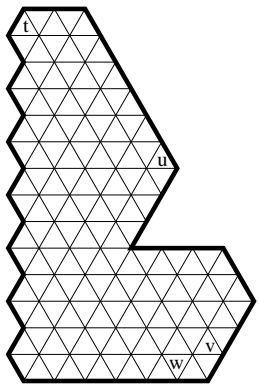}
 \caption{Applying Kuo condensation to determine $M(R_{a,k,j,x}),$ with $a=3,k=4,j=6,\textrm{and }x=2$.}
 \label{rkuo}
 \end{figure}

 Applying Kuo condensation to such a region gives us a recurrence involving six new regions.  They are shown in Figure~\ref{recurrencepic}.  In each subfigure, the triangles corresponding to removed vertices are labelled and any subsequently forced lozenges are shown.
 
 \begin{figure}
 \centering
 \includegraphics[height=7in]{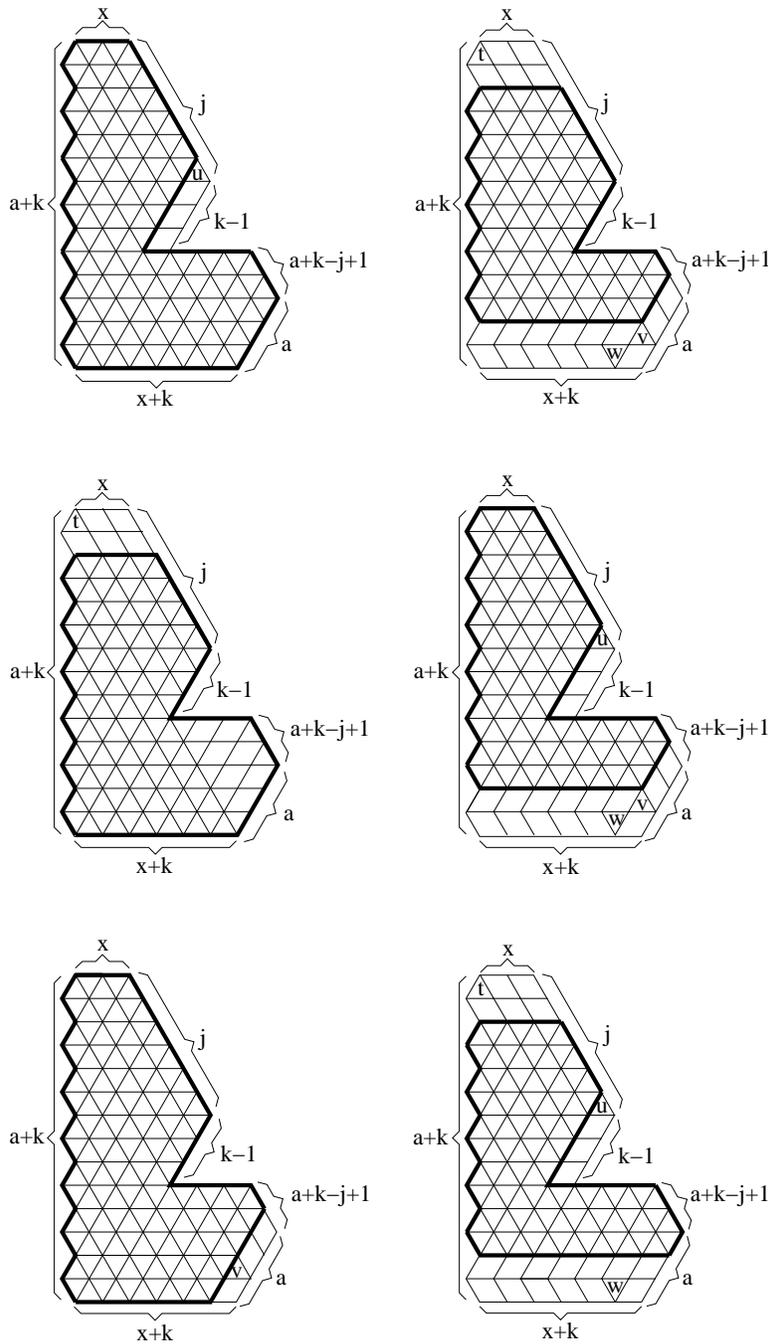}
 \caption{The regions obtained via Kuo condensation for $M(R_{a,k,j,x})$.}
 \label{recurrencepic}
 \end{figure}

 It is evident that 
 \begin{align*}
 M(G-u) &=  M(R_{a,k,j,x}),\\
 M(G-\{t,v,w\}) &=M(R_{a-1,k-1,j-1,x+1}),\\
 M(G-t) &=M(R_{a,k-1,j-1,x+1}),\\
 M(G-\{u,v,w\}) &=M(R_{a-1,k,j,x}),\\
 M(G-v) &=M(R_{a+1,k-1,j+1,x}),\textrm{ and }\\
 M(G-\{t,u,w\}) &=M(R_{a-2,k,j-2,x+1}).
 \end{align*}
 
 Therefore, it must be the case that
 \begin{multline}\label{recurrence}
 M(R_{a,k,j,x})M(R_{a-1,k-1,j-1,x+1})=\\
 M(R_{a,k-1,j-1,x+1})M(R_{a-1,k,j,x})+M(R_{a+1,k-1,j+1,x})M(R_{a-2,k,j-2,x+1}).
 \end{multline}
 
 At this point, we proceed by induction on $a$.  We can rewrite (\ref{recurrence}) as
 
 \begin{multline}\label{recurrencerewrite}
 M(R_{a+1,k-1,j+1,x})=\\
  \dfrac{M(R_{a,k,j,x})M(R_{a-1,k-1,j-1,x+1})-M(R_{a,k-1,j-1,x+1})M(R_{a-1,k,j,x})}{M(R_{a-2,k,j-2,x+1})},
 \end{multline}
so that the region on the lefthand side has southeast side-length $a+1$, while all those on the right are shorter, ranging from $a-2$ to $a$. Since $j>k+1$, $j-2$ is at least $k$ so that the $j$ index falls into the proper range.  In Figure~\ref{cantile}, we break up $R_{a-2,k,j-2,x+1}$ into five regions: three parallelograms and two Proctor regions.  This shows that $M(R_{a-2,k,j-2,x+1})\neq 0$ since the parallelograms each have a unique tiling and \Cref{proctorthm} proves that the other regions have tilings.  If we show that the formula in \Cref{rakjxthm} satisfies (\ref{recurrence}) or (\ref{recurrencerewrite}) and that the formula holds for $a=0,1,2$, we will have proven the desired result.  
 
 \begin{figure}
 \centering
 \includegraphics[height=2in]{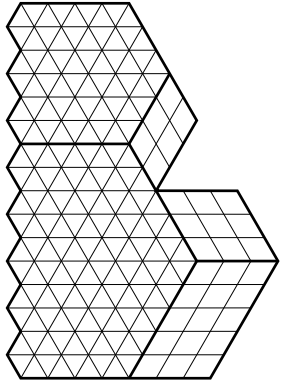}
 \caption{Breaking up $R_{a-2,k,j-2,x+1}$ in this way shows that a tiling exists.}
 \label{cantile}
 \end{figure}
 
 We will show that (\ref{recurrence}) holds for $a>2, k>0, j>2,$ and $x\geq 0.$ Substituting our claimed formula from \Cref{rakjxthm} into (\ref{recurrence}) and rearranging terms, we need to show:
 
 \begin{multline}\label{pluggedinrecurrence}
  \left[M(P_{j-1,j-1,x})M(P_{j-2,j-2,x+1})\right]\cdot\left[M(P_{a-j+k+1,a,k})M(P_{a-j+k,a-1,k-1})\right]\\
 \times\dfrac{\left[Q_{a,k,x}Q_{a-1,k-1,x+1}\right]\cdot\left[Q_{j-k-1,k,0}Q_{j-k-1,k-1,0}\right]}{\left[Q_{a,k,0}
 Q_{a-1,k-1,0}\right]\cdot\left[Q_{j-k-1,k,x}Q_{j-k-1,k-1,x+1}\right]}\\
  =\left[M(P_{j-2,j-2,x+1})M(P_{j-1,j-1,x})\right]\cdot\left[M(P_{a-j+k+1,a,k-1})M(P_{a-j+k,a-1,k})\right]\\
 \times \dfrac{\left[Q_{a,k-1,x+1}Q_{a-1,k,x}\right]\cdot\left[Q_{j-k-1,k-1,0}Q_{j-k-1,k,0}\right]}{\left[Q_{a,k-1,0}Q_{a-1,k,
 0}\right]\cdot\left[Q_{j-k-1,k-1,x+1}Q_{j-k-1,k,x}\right]}\\
  +\left[M(P_{j,j,x})M(P_{j-3,j-3,x+1})\right]\cdot\left[M(P_{a-j+k,a+1,k-1})M(P_{a-j+k+1,a-2,k})\right]\\
  \times \dfrac{\left[Q_{a+1,k-1,x}Q_{a-2,k,x+1}\right]\cdot\left[Q_{j-k+1,k-1,0}Q_{j-k-3,k,0}\right]}{\left[Q_{a+1,k-1,0}
  Q_{a-2,k,0}\right]\cdot\left[Q_{j-k+1,k-1,x}Q_{j-k-3,k,x+1}\right]},\\
 \end{multline}
 
where $M(P_{a,b,c})$ are given by \Cref{proctorthm}.  We've broken up each of the three terms in this equation into six parts: two of the parts are products of matching generating functions of Proctor regions, and the other four parts are products of $Q$ polynomials.  We select corresponding parts of the three terms and simplify them.  Here we will show only the simplification process for two of the six parts.  Combining the results together will prove (\ref{pluggedinrecurrence}).

First we consider the parts which are products of matching generating functions of Proctor regions whose first two indices are equal:
	\begin{multline*}
	M(P_{j-1,j-1,x})M(P_{j-2,j-2,x+1}), \textrm{  }M(P_{j-2,j-2,x+1})M(P_{j-1,j-1,x}),\\
	M(P_{j,j,x})M(P_{j-3,j-3,x+1}).
	\end{multline*}
	Dividing by $M(P_{j-1,j-1,x})M(P_{j-2,j-2,x+1})$ makes the first two products both $1$.  For the third, notice that 
	$$\frac{M(P_{j,j,x})}{M(P_{j-1,j-1,x})}=\frac{(x+j)}{(2x+j)}\cdot\prod_{i=1}^{j}\frac{2x+j-1+i}{j-1+i}.$$

	Applying this twice and simplifying shows that $$\frac{M(P_{j,j,x})M(P_{j-3,j-3,x+1})}{M(P_{j-1,j-1,x})M(P_{j-2,j-2,x		+1})}=\frac{(x+j)(2x+2j-1)}{2(2j-3)(2j-1)},$$ so that these products simplify, in order, to $$1,\textrm{  }1,\textrm{  }\frac{(x+j)(2x+2j-1)}		{2(2j-3)(2j-1)}.$$

Now we simplify one of the parts consisting of $Q$ polynomials: 
	$$Q_{a,k,x} Q_{a-1,k-1,x+1}, \textrm{  }Q_{a,k-1,x+1} Q_{a-1,k,x}, \textrm{  }Q_{a+1,k-1,x} Q_{a-2,k,x+1}.$$
	Let us divide by $Q_{a,k,x} Q_{a-1,k-1,x+1}.$  The second factor is equal to the first.  Further 
	$$\dfrac{Q_{a+1,k-1,x}}{Q_{a,k,x}}=\prod_{i=1}^{\lceil\frac{a+1}{2}\rceil} (x+k+i-1)\prod_{i=1}^{\lceil\frac{a}{2}\rceil} 		(2x+2k+2i-1).$$
	 We can apply this twice and simplify to get $$1,\textrm{  } 1, \textrm{  }(x+k)(2x+2k+1).$$

 Here are the results obtained when simplifying the other four parts.
 
 \begin{itemize}
 \item Divide each of 
 \begin{multline*}M(P_{a-j+k+1,a,k})M(P_{a-j+k,a-1,k-1}),\\
  \textrm{  }M(P_{a-j+k+1,a,k-1})M(P_{a-j+k,a-1,k}), \\\textrm{and }M(P_{a-j+k,a		+1,k-1})M(P_{a-j+k+1,a-2,k})
	\end{multline*}
 by the first product and rearrange factors. These simplify to $$\frac{(3k+2a-j-1)(3k+2a-j)}{(2k+a)(2k+a-1)}, \textrm{  }\frac{(2k+a-j)}{(k+a)},\textrm{  }\frac{(j-k-1)(j-k)}{j(j-1)}.$$
 
 \item The three products 
 \begin{multline*}Q_{j-k-1,k,0}Q_{j-k-1,k-1,0}, \textrm{  }Q_{j-k-1,k-1,0}Q_{j-k-1,k,0},\\ \textrm{and }Q_{j-k+1,k-1,0}Q_{j-k-3,k,0}
 \end{multline*} simplify to $$1,\textrm{  } 1,\textrm{  } (j-1)j(2j-3)(2j-1)$$ when divided by the first.
 
 \item We divide $$[Q_{a,k,0} Q_{a-1,k-1,0}]^{-1},\textrm{  } [Q_{a,k-1,0}Q_{a-1,k,0}]^{-1}, \textrm{ and }[Q_{a+1,k-1,0}Q_{a-2,k,0}]^{-1}$$ each by the middle product, and rearrange some factors.  The simplification leads to $$(k+a-\lfloor\frac{a+1}{2}\rfloor)(2k+2a-2\lfloor\frac{a}{2}\rfloor-1), \textrm{  }(k+1)(2k			+2a-1), \textrm{  }1.$$
 
 \item The final three products, 
 	\begin{multline*} [Q_{j-k-1,k,x}Q_{j-k-1,k-1,x+1}]^{-1}, \textrm{  }[Q_{j-k-1,k-1,x+1}Q_{j-k-1,k,x}]^{-1}, \\ \textrm{and }[Q_{j-k+1,k-1,x}Q_{j-		k-3,k,x+1}]^{-1}
	\end{multline*} can be reduced to $$1, \textrm{  } 1, \textrm{ and } \frac{1}{(x+k)(x+j)(2x+2k+1)(2x+2j-1)}$$ if we divide by the first.
 
 \end{itemize}
 
 Combining these results (and reducing fractions) simplifies the proof of (\ref{recurrence}) to verifying that the following equation holds:
 
 \begin{multline}\label{almost}
\frac{(3k+2a-j-1)(3k+2a-j)}{(2k+a)(2k+a-1)}\cdot (k+a-\lfloor\frac{a+1}{2}\rfloor)(2k+2a-2\lfloor\frac{a}{2}\rfloor -1) \\
= \frac{(2k+a-j)}{(k+a)}\cdot(k+1)(2k+2a-1)+\frac{(j-k-1)(j-k)}{2}.
 \end{multline}
 
 It is easy to see that (\ref{almost}) is true, showing that the claimed formula from \Cref{rakjxthm} satisfies the recurrence implied by \Cref{kuograph}.

We now prove that \Cref{rakjxthm} holds for $a=0,1,2,$ assuming $k>0.$  Based upon the values that $j$ can take, we have three cases:
\begin{enumerate}
\item $a=1 \textrm{ and } j=k+2$,
\item $a=2 \textrm{ and } j=k+2$,
\item $a=2 \textrm{ and } j=k+3$.
\end{enumerate}

In each of these three cases we need to show that the formula in \Cref{rakjxthm} holds.  
In cases (1) and (3) all of the Q polynomials cancel, and $M(P_{a-j+k+1,a,k})=1$ because at least one of the indices is 0.  Therefore, we only need to check that the $R$ region and remaining Proctor region in the formula in \Cref{rakjxthm} have the same number of tilings.  This is easily accomplished since in each case some lozenges in the $R$ region are forced, making the two regions essentially identical.  Case (3) follows by Lemma 4.3(a) in \cite{ciucufischer13}.

Finally, if $k=0$ we remove no unit triangles from the northeast side.  For the region to be balanced, it must be the case that this side has length $a$, just as the southeast side does.  Thus, we must have $M(R_{a,0,j,x})=M(P_{a,a,x}).$  Substituting the claimed formula from \Cref{rakjxthm} implies that we need to show

\begin{equation}\label{kzero}
M(P_{a,a,x})=\frac{Q_{a,0,x}}{Q_{a,0,0}}\cdot\frac{Q_{j-1,0,0}}{Q_{j-1,0,x}}\cdot M(P_{j-1,j-1,x})\cdot M(P_{a-j+1,a,0}).
\end{equation}

It is clear that $M(P_{a-j+1,a,0})=1$.  Using \Cref{jisk2} with $k=0$ and $a=j-1$, we see that 

\[M(P_{j-1,j-1,x})=\frac{Q_{j-1,0,x}}{Q_{j-1,0,0}},\]
as $M(P_{j-1,j-1,0})=1.$
This reduces (\ref{kzero}) to
\[M(P_{a,a,x})=\frac{Q_{a,0,x}}{Q_{a,0,0}},\]
which is proven again by \Cref{jisk2} with $k=0.$
 
 

The index $j$ drops out of the lefthand side of (\ref{kzero}) entirely after applying \Cref{jisk2} to $M(P_{j-1,j-1,x})$.  If we think of $M(P_{a,b,c})$ as an expression involving integers $a,b,$ and $c$, (which need not arise from a realizable region $P_{a,b,c}$), then $j$ can take any integer value when $k=0$.  
\end{proof}

The proof of \Cref{weightedrakjxthm} is similar to that of \Cref{rakjxthm}.  \Cref{kuograph} is applied identically - the locations of $t,u,v,\textrm{ and } w$ are the same.  The resulting recurrence is the same as (\ref{recurrence}), with $R$ replaced by $R'$.  This only holds because none of the forced lozenges in Figure~\ref{recurrencepic} are any of those which are weighted by $\frac{1}{2}$ in $R'$.  The processes of verifying that the recurrence and base cases hold is analogous to the work done above.
 
 \end{document}